\documentclass[12pt,reqno]{amsart}
\usepackage{graphicx, amssymb, color,pdfsync}

\numberwithin{equation}{section}

\advance\hoffset by -0.48 truecm

\newtheorem{theo}{Theorem}

\newtheorem{lemm}[theo]{Lemma}

\theoremstyle{remark}

\newcommand{\s}{\vspace{0.3cm}}

\begin{document}

\title{Fields of moduli of classical Humbert Curves}

\author[R.A. Hidalgo]{Rub\'en A. Hidalgo}
\address{Departamento de Matem\'atica, Universidad T\'ecnica Federico Santa Mar\'{\i}a, Casilla 110-V Valparaiso, Chile}
\email{ruben.hidalgo@usm.cl, sebastian.reyes@usm.cl}

\author[S. Reyes-Carocca]{Sebastian Reyes-Carocca}

\thanks{Partially supported by project Fondecyt 1070271 and UTFSM 12.09.02}


\begin{abstract}
The computation of the field of moduli of a closed Riemann surface seems to be a very difficult problem and even more difficult is to determine if the field of moduli is a field of definition. In this paper we consider the family of closed Riemann surfaces of genus five admitting a group of conformal automorphisms isomorphic to ${\mathbb Z}_{2}^{4}$. These surfaces are non-hyperelliptic ones and turn out to be the highest branched abelian covers of the orbifolds of genus zero and five cone points of order two. We compute the field of moduli of these surfaces and we prove that they are fields of definition. This result is in contrast with the case of the highest branched abelian covers of the orbifolds of genus zero and six cone points of order two as there are cases for which the above property fails.

\end{abstract}

\maketitle

%%%%%%%%%%%%%%%%
%%%%%%%%%%%%%%%%
\section{Introduction}
Let us denote by ${\rm Aut}({\mathbb C}/{\mathbb Q})$ the group of field automorphisms of the complex number field ${\mathbb C}$. If $P$ is some polynomial with complex coefficients and $\sigma \in {\rm Aut}({\mathbb C}/{\mathbb Q})$, then we denote by $P^{\sigma}$ the new polynomial obtained by applying $\sigma$ to the coefficients of $P$. If $C$ is some projective non-singular algebraic curve, say defined by polynomials $P_{1}$,..., $P_{r}$, and $\sigma \in {\rm Aut}({\mathbb C}/{\mathbb Q})$, then we denote by $C^{\sigma}$ the projective non-singular algebraic curve defined by the polynomials $P_{1}^{\sigma}$,..., $P_{r}^{\sigma}$. In this way, a natural action of ${\rm Aut}({\mathbb C}/{\mathbb Q})$ on the category of projective non-singular algebraic curves is obtained. As a complex projective non-singular algebraic curve defines a closed Riemann surface and,  as a consequence of the Riemann-Roch theorem, a closed Riemann surface can be described by a complex projective non-singular algebraic curve, there is a natural action of ${\rm Aut}({\mathbb C}/{\mathbb Q})$ at the level of closed Riemann surfaces. If $C_{1}$ and $C_{2}$ are projective non-singular algebraic curves, which are conformally equivalent as Riemann surfaces (we denote this by $C_{1} \cong C_{2}$) and $\sigma \in {\rm Aut}({\mathbb C}/{\mathbb Q})$, then it is known that $C_{1}^{\sigma} \cong C_{2}^{\sigma}$. In particular, we obtain a natural action of ${\rm Aut}({\mathbb C}/{\mathbb Q})$ on the moduli space of closed Riemann surfaces on each genera. This action is not well understood in genus at least $2$. 

Let $S$ be a closed Riemann surface. If $C$ is any projective non-singular algebraic curve which is, as a Riemann surface, conformally equivalent to $S$, then the stabilizer of the conformal class of $S$ is the group $K_{S}=\{\sigma \in {\rm Aut}({\mathbb C}/{\mathbb Q}): C^{\sigma} \cong C\}$. The fixed field of $K_{S}$, say  ${\mathcal M}(S)$, is known as the {\it field of moduli of $S$}. By the definition, the above does not depends on the specific choice of $C$ as longer it defines $S$. A {\it field of definition} of the closed Riemann surface $S$ is any subfield ${\mathbb F}<{\mathbb C}$ for which there is a finite collection of homogeneous polynomials, each of them defined over ${\mathbb F}$, so that they define a complex projective non-singular algebraic curve which is, view as a Riemann surface, conformally equivalent to $S$. 

It is well known that the field of moduli ${\mathcal M}(S)$ is contained in any field of definition of $S$ and Koizumi proved \cite{Koizumi1} that it coincides with the intersection of all fields of definitions of it. D\`ebes-Emsalem \cite{DE} (see also Hammer-Herrlich \cite{HH}) proved that there is a field of definition of $S$ which is an extension of finite degree of its field of moduli.

Both, the computation of the field of moduli and the determination of whether the field of moduli  is a field of definition are difficult problems for the case that $S$ has genus $g \geq 2$. 

Necessary conditions for $S$ to be definable over its field of moduli were provided by Weil \cite{Weil}. Weil's conditions hold trivially if ${\rm Aut}(S)$ is trivial. On the other extreme, if $S/{\rm Aut}(S)$ is an orbifold with signature of type $(0;a,b,c)$, $S$ is called quasiplatonic,  then Wolfart \cite{Wolfart} proved that $S$ can be defined over its field of moduli.

On 1972, explicit examples of hyperelliptic Riemann surfaces, which cannot be defined over their fields of moduli, were provided separately by Earle and Shimura  \cite{Earle, Shimura}.  Recently, on 2007, Huggins \cite{Huggins2} proved that a hyperelliptic Riemann surface $S$, with hyperelliptic involution $\iota$, for which ${\rm Aut}(S)/\langle \iota \rangle$ is neither trivial nor cyclic, can be defined over its field of moduli. Moreover, in the same paper, for each $n$ so that ${\rm Aut}(S)/\langle \iota \rangle \cong {\mathbb Z}_{n}$ there is constructed an example which cannot be definable over its field of moduli.

 In \cite{Earle} Earle stated the existence of non-hyperelliptic Riemann surfaces which cannot be defined over their fields of moduli, but no explicit example was provided. On 2009, explicit examples of non-hyperelliptic Riemann surfaces which cannot be definable over their fields of moduli were constructed by the first author  \cite{Hid}. These non-hyperelliptic examples turns out to be highest regular branched covers of an orbifold with signature $(0;2,2,2,2,2,2)$. These examples are also ${\mathbb Z}_{2}^{3}$ covers of the examples provided by Earle in \cite{Earle}.

A highest regular branched cover of an orbifold with signature $(0;2,\stackrel{n}{\ldots},2)$ is called a {\it generalized Humbert curve of type $n-1$} and it has genus  at least two if and only if $n \geq 5$; in which case it is known to be a non-hyperelliptic Riemann surface \cite{CGHR}. In this notation, the examples in \cite{Hid} are generalized Humbert curves of type $5$. A generalized Humbert curve of type $4$ is called a {\it classical Humbet curve}; these are closed Riemann surfaces of genus five and they were firstly studied by Humbert \cite{Humbert} and lately by Edge \cite{Edge} and Varley \cite{Varley} among others. 

We were wondering if examples of classical Humbert curves which cannot be definable over their fields of moduli may exist. In this paper we provide a negative answer, that is, we prove that each of these Riemann surfaces can be defined over their fields of moduli. Moreover, we compute the corresponding field of moduli.

The proof of the main result just follow the same techniques developed by D\`ebes-Emsalem in \cite{DE} (see also Section \ref{Sec:prelim}) adapted to our concrete family of curves.

%%%%%%%%%%%%%%%
\section{Classical Humbert curves and main results}
Let $S$ be a classical Humbert curve. By the definition,  $S$ admits  groups of conformal automorphisms  isomorphic to ${\mathbb Z}_{2}^{4}$;  called {\it classical Humbert groups} of $S$. 

If $H$ is a classical Humbert group of $S$, then Edge \cite{Edge} noted that each fixed point of a non-trivial element of $H$ is a Weierstrass point and, in particular, they are all the Weierstrass points of $S$. It follows from this that a classical Humbert curve admits at most one classical Humbert group and that $S$ is a non-hyperelliptic Riemann surface.  If  $H$ is  the classical Humbert group of $S$, then the pair $(S,H)$ is called a  \textit{classical Humbert  pair}.

The uniqueness of $H$ asserts that it is a normal subgroup of $Aut(S)$; so ${\rm Aut}(S)/H$ can be seen as a subgroup of $Aut_{orb}(S/H)$, where $Aut_{orb}(S/H)$ denotes the group of conformal automorphisms of the orbifold $S/H$. Now, as $S/H$ is an orbifold with signature $(0;2,2,2,2,2)$, it follows from the classical uniformization theorem that there is a Fuchsian group $\Gamma$ with presentation 
$$
\Gamma=\langle x_{1},...,x_{5}: x_{1}^{2}=\cdots=x_{5}^{2}=x_{1}x_{2}\cdots x_{5}=1\rangle,
$$
so that ${\mathbb H}^{2}/\Gamma=S/H$. It is not difficult to see that ${\mathbb H}^{2}/\Gamma' = S$ and that $H=\Gamma/\Gamma'$, where $\Gamma'$ is the derived subgroup of $\Gamma$. As $\Gamma'$ is unique in $\Gamma$, it follows that 
 ${\rm Aut}(S)/H=Aut_{orb}(S/H)$ and, in particular, that $Aut(S)$ can be obtained by the lifting to $S$ of $Aut_{orb}(S/H)$. As the orbifold $S/H$ is the Riemann sphere with exactly $5$ conical points, all of them of order $2$, we may identify $Aut_{orb}(S/H)$ with the finite subgroup of the group of M\"obius transformations
${\mathbb M}$ keeping invariant these $5$ conical points (see also \cite{CGHR}).

Now, as a classical Humbert curve $S$ is non-hyperelliptic, it may be holomorphically embedded as an smooth complex projective algebraic curve of degree $8$, say $C \subset {\mathbb P}_{\mathbb C}^{4}$, by using a basis of the $5$-dimensional complex space of holomorphic one-forms $H^{1,0}(S)$; the canonical curve \cite{Farkas-Kra}. In fact, Humbert \cite{Humbert} provides such a canonical curve as the intersection of $3$ (diagonal) quadrics in ${\mathbb P}_{\mathbb C}^{4}$. One of these three quadrics uses all the coordinates, that is, it is of the form $x_{1}^{2}+x_{2}^{2}+x_{3}^{2}+x_{4}^{2}+x_{5}^{2}=0$. 

In \cite{CGHR} a new holomorphic embedding of $S$, as an smooth complex projective algebraic curve of degree $8$ in ${\mathbb P}_{\mathbb C}^{4}$ was obtained as 
$$
C_{\lambda_{1},\lambda_{2}}=
\left\{\begin{array}{c}
x_{1}^{2}+x_{2}^{2}+x_{3}^{2}=0\\
\lambda_{1}x_{1}^{2}+x_{2}^{2}+x_{4}^{2}=0\\
\lambda_{2}x_{1}^{2}+x_{2}^{2}+x_{5}^{2}=0
\end{array}
\right.
$$
where, up to a M\"obius transformation, the $5$ conical points of $S/H$ are given by $\infty$, $0$, $1$, $\lambda_{1}$ and $\lambda_{2}$, so that $\lambda_{1} \neq \lambda_{2}$ and $\lambda_{1},\lambda_{2} \in {\mathbb C}-\{0,1\}$. In this representation the classical Humbert group
$H$ corresponds to the group generated by the involutions $a_{1}$,...,$a_{4}$, where $a_{j}$ is multiplication by $-1$ in the $j$-th coordinate and 
$\pi:C_{\lambda_{1},\lambda_{2}} \to \widehat{\mathbb C}$, defined by
$\pi([x_{1}:x_{2}:x_{3}:x_{4}:x_{5}])=-(x_{2}/x_{1})^{2}$, is a Galois cover with $H$ as cover group whose branch values are $\infty$, $0$, $1$, $\lambda_{1}$ and $\lambda_{2}$.

The uniqueness of $H$ asserts that $C_{\lambda_{1},\lambda_{2}} \cong C_{\mu_{1},\mu_{2}}$ if and only if there is some $T \in {\mathbb G}$ so that $T(\lambda_{1},\lambda_{2})=(\mu_{1},\mu_{2})$, where
$$
{\mathbb G}=\langle A(z,w)=(1/z, 1/w), B(z,w)=(w/(w-1), w/(w-z)\rangle \cong {\mathfrak S}_{5}.
$$

Given a classical Humbert curve $C_{\lambda_{1},\lambda_{2}}$ and 
given $\sigma \in {\rm Aut}({\mathbb C}/{\mathbb Q})$, we have the new classical Humbert curve $C_{\lambda_{1},\lambda_{2}}^{\sigma}=C_{\sigma(\lambda_{1}),\sigma(\lambda_{2})}$. In general, $C_{\lambda_{1},\lambda_{2}}$ and $C_{\sigma(\lambda_{1}),\sigma(\lambda_{2})}$ are not conformally equivalent Riemann surfaces. Let us denote by $K_{\lambda_{1},\lambda_{2}}<{\rm Aut}({\mathbb C}/{\mathbb Q})$ the subgroup consisting of all those $\sigma \in {\rm Aut}({\mathbb C}/{\mathbb Q})$ for which $C_{\lambda_{1},\lambda_{2}}$ and $C_{\sigma(\lambda_{1}),\sigma(\lambda_{2})}$ are conformally equivalent Riemann surfaces. In this way, the fixed field of 
$K_{\lambda_{1},\lambda_{2}}$ is the field of moduli ${\mathcal M}(C_{\lambda_{1},\lambda_{2}})$ of $C_{\lambda_{1},\lambda_{2}}$. 

The main result of this paper is the following.

\s
\noindent
\begin{theo}\label{theo1}
Let $\lambda_{1}, \lambda_{2} \in {\mathbb C}-\{0,1\}$ be so that $\lambda_{1} \neq \lambda_{2}$. Then
\begin{enumerate}
\item ${\mathcal M}(C_{\lambda_{1},\lambda_{2}})={\mathbb Q}(j_{1}(\lambda_{1},\lambda_{2}),j_{2}(\lambda_{1},\lambda_{2}))$, where 
$j_{1}$ and $j_{2}$ are provided in the proof.

\item ${\mathcal M}(C_{\lambda_{1},\lambda_{2}})$ is a field of definition for $C_{\lambda_{1},\lambda_{2}}$.
\end{enumerate}
\end{theo}

\s

%%%%%%%%%%%%%%%%%%%%%%%
%%%%%%%%%%%%%%%%%%%%%%%
%%%%%%%%%%%%%%%
\section{D\`ebes-Emsalem's method}\label{Sec:prelim}
 In \cite{DE,HH} it was proved that a non-singular projective algebraic curve can be defined over a  finite Galois extension of its field of moduli.
 This observation and Weil's theorem \label{Weil} provide sufficient conditions for the field of moduli to be a field of definition. We state Weil's theorem in a suitable form adequate for us.

\s
\noindent
\begin{theo}[Weil \cite{Weil}]\label{Prop:Weil}
Let $C$ be a non-singular projective algebraic curve defined over a finite Galois extension $L$ of its field of moduli ${\mathcal M}(C)$. 
If for every $\sigma \in {\rm Aut}(L/{\mathcal M}(C))$ there is a biholomorphism $f_{\sigma}:C \to C^{\sigma}$ defined over $L$ such that for all $\sigma, \tau \in {\rm Aut}(L/{\mathcal M}(C))$ the compatibility condition $f_{\tau\sigma}=f^{\tau}_{\sigma} \circ f_{\tau}$
holds, then there exists a non-singular projective algebraic curve $E$ defined over ${\mathcal M}(C)$ and there exists a biholomorphism $R:C \to E$, defined over $L$, such that $R^{\sigma} \circ f_{\sigma}=R$.
\end{theo}

\s

Next result, due to D\`ebes-Emsalem \cite{DE}, is a consequence of Weil's theorem and provides a sufficient condition if a curve can be defined over its field of moduli. This method is given by the computation of certain points. First, we need to recall some definitions.
Let us consider a (branched) holomorphic covering between non-singular projective algebraic curves (so closed Riemann surfaces) $S$ and $R$, say 
$f:S \to R$, and let us assume that $R$ is defined over ${\mathcal M}(S)$. For each $\sigma \in {\rm Aut}({\mathbb C}/{\mathcal M}(S))$ we may consider the (branched) holomorphic covering $f^{\sigma}:S^{\sigma} \to R^{\sigma}=R$. We say that they are equivalent, noted as $\{f^{\sigma}:S^{\sigma} \to R\} \cong \{f:S \to R\}$, if there is a biholomorphism $\phi_{\sigma}:S \to S^{\sigma}$ so that $f^{\sigma} \circ \phi_{\sigma}=f$. The {\it field of moduli} of $f:S \to R$, denoted by ${\mathcal M}(f:S \to R)$,  is the fixed field of the subgroup 
$$G(f:S \to R)=\left\{\sigma \in {\rm Aut}({\mathbb C}/{\mathcal M}(S)): \{f^{\sigma}:S^{\sigma} \to R\} \cong \{f:S \to R\}\right\}<{\rm Aut}({\mathbb C}/{\mathbb Q}).$$

\s
\noindent
\begin{theo}[D\`ebes-Emsalem \cite{DE}]\label{canonico}
If $S$ is a non-singular projective algebraic curve of genus $g \geq 2$, then there exists a non-singular projective algebraic curve $B$, defined over ${\mathcal M}(S)$, and there exists 
a Galois cover $f:S \to B$, with ${\rm Aut}(S)$ as Deck group, so that ${\mathcal M}(f:S \to B)={\mathcal M}(S)$. Moreover, if $B$ contains at least one ${\mathcal M}(S)$-rational point outside the branch locus of $f$, then ${\mathcal M}(S)$ is also a field of definition of $S$. Such a curve $B$ is called a {\it canonical model} of $S/{\rm Aut}(S)$.
\end{theo}

\s

We proceed to recall the general arguments for the proof of the above. Let $S$ be a non-singular projective algebraic curve (a closed Riemann surface) of genus at least two. We denote by ${\mathcal O}$ the Riemann orbifold $S/{\rm Aut}(S)$ and by ${\mathcal O}_{B}$ its subset of cone points. 
Let $P:S \to {\mathcal O}$ be a Galois cover with ${\rm Aut}(S)$ as its group of Deck transformations. For each $\sigma \in K_{S}$, where $K_{S}=\{\sigma \in {\rm Aut}({\mathbb C}/{\mathbb Q}): S^{\sigma} \cong S\}$, there is a biholomorphism
$f_{\sigma}:S \to S^{\sigma}$. It follows the existence of an automorphism of the orbifold ${\mathcal O}$, say $g_{\sigma}$, uniquely determined by $\sigma$, so that $P^{\sigma} \circ f_{\sigma}=g_{\sigma} \circ P$. The uniqueness of the automorphisms $g_{\sigma}$ ensures that the collection $\{g_{\sigma}: \sigma \in G_{S}\}$ satisfies the conditions on Weil's theorem for the underlying Riemann surface structure of ${\mathcal O}$, say $X$. It follows the existence of an irreducible projective algebraic curve $B$ defined over ${\mathcal M}(S)$ and a biholomorphism $Q:X \to B$ so that $Q=Q^{\sigma} \circ g_{\sigma}$. In this case, $f=Q \circ P:S \to B$ satisfies that ${\mathcal M}(f:S \to B)={\mathcal M}(S)$. Let us now assume that there is a point $q \in B$ which is not a branch values of $f$ and which is ${\mathcal M}(S)$-rational point. Let us fix some point $p \in S$ so that $f(p)=q$. One may check that it is posible, for each $\sigma \in K_{S}$, to choose $f_{\sigma}:S \to S^{\sigma}$ so that $f_{\sigma}(p)=\sigma(p)$. In fact, we first note that $f^{\sigma}(f_{\sigma}(p))=f(p)=q$ and that $f^{\sigma}(\sigma(p))=\sigma(f(p))=\sigma(q)=q$. In this way, there is some $h_{\sigma} \in {\rm Aut}(S^{\sigma})$ so that $h_{\sigma}(f_{\sigma}(p))=\sigma(p)$ and we may replace $f_{\sigma}$ by $h_{\sigma} \circ f_{\sigma}$ in order to have the required property. Next, one checks that such collection of new biholomorphism satisfies Weil's conditions to obtain that $S$ can be defined over its field of moduli. 

Now, in order to compute ${\mathcal M}(S)$-rational points we make use of the following fact.

\s
\noindent
\begin{lemm}\label{calculo}
In the above notation, $S$ can be defined over its field of moduli if there exists a point $r \in {\mathcal O}-{\mathcal O}_{B}$ so that
$$\sigma(r)=g_{\sigma}(r), \; \forall \sigma \in K_{S}.$$
\end{lemm}
\begin{proof}
The existence of a point $q \in B$ which is ${\mathcal M}(S)$-rational and not a branch value for the branched covering $f:S \to B$ is equivalent to find a point $r \in {\mathcal O}-{\mathcal O}_{B}$ (with $q=Q(r)$) with the property that
$$\sigma(Q(r))=Q(r), \; \forall \sigma \in K_{S}.$$

As, for $\sigma \in K_{S}$,  $\sigma(Q(r))=Q^{\sigma}(\sigma(r))=Q \circ g_{\sigma}^{-1}(\sigma(r))$,  the above is equivalent to find a point $r \in {\mathcal O}-{\mathcal O}_{B}$ with the property that
$$\sigma(r)=g_{\sigma}(r), \; \forall \sigma \in K_{S}.$$
\end{proof}

\s

The equality given in Lemma \ref{calculo}  permits to obtain the existence of a ${\mathcal M}(S)$-rational point in $B$ without knowing explicitly $B$; for it, we only need to know the transformations $g_{\sigma}$. Next, we provide the way how to compute such value of $r$ with only a finite number of computations. 

First, we states the following easy fact. Set $U=\{g_{\sigma}: \sigma \in K_{S}\}$.

\s
\noindent
\begin{lemm}\label{homo}
If $g_{\sigma}$ is defined over ${\mathcal M}(S)$, for every $\sigma \in K_{S}$, then $U$ is a group and 
$\Phi:K_{S} \to {\mathbb M}$, defined as $\Phi(\sigma)=g_{\sigma}^{-1}$, is a homomorphism of groups.
\end{lemm}
\begin{proof}
Let $\sigma, \tau \in K_{\lambda}$. It is not hard to see that the following equality holds
$$
P^{\sigma \tau} \circ f_{\tau}^{\sigma} \circ f_{\sigma} = g_{\tau}^{\sigma} \circ g_{\sigma} \circ P.
$$

As $f_{\tau}^{\sigma} \circ f_{\sigma}=f_{\sigma \tau} \circ h$, for a suitable $h \in {\rm Aut}(C_{\lambda})$, and $g_{\sigma \tau} \in {\mathbb M}$ is uniquely determined by $\sigma \tau$ and satisfies the equality $P^{\sigma \tau} \circ f_{\sigma \tau}=g_{\sigma \tau} \circ P$, 
we obtain that $g_{\sigma \tau}=g_{\tau}^{\sigma} \circ g_{\sigma}$.  Since, in our particular case,  $g_{\tau}^{\sigma}=g_{\tau}$, the above asserts that $g_{\sigma \tau}=g_{\tau} \circ g_{\sigma}$ and we are done. 
\end{proof}

\s
\noindent
\begin{lemm}\label{genera}
Let us assume that, for every $\sigma \in K_{S}$ we have that $g_{\sigma}$ is defined over ${\mathcal M}(S)$.
Let $\sigma_{1},...,\sigma_{n} \in K_{S}$ be so that $g_{\sigma_{1}}$,..., $g_{\sigma_{n}}$ generate $U=\Phi(K_{S})$. Assume that there is some $r \in {\mathbb C}-\{0,1,\lambda\}$ so that
$g_{\sigma_{j}}(r)=\sigma_{j}(r)$, for every $j=1,...,n$. Then $g_{\sigma}(r)=\sigma(r)$, for every $\sigma \in K_{S}$.
\end{lemm}
\begin{proof}
We only need to prove that $\sigma_{i}(\sigma_{j}(r))=g_{\sigma_{i}\sigma_{j}}(r)$, for every $i,j \in \{1,...,n\}$. But,
$$\sigma_{i}(\sigma_{j}(r))=\sigma_{i}(g_{\sigma_{j}}(r))=g_{\sigma_{j}}^{\sigma_{i}}(\sigma_{i}(r))=g_{\sigma_{j}}(\sigma_{i}(r))=g_{\sigma_{j}}(g_{\sigma_{i}}(r))=g_{\sigma_{j}} \circ g_{\sigma_{i}}(r).$$

Again, as $\Phi$ is a homomorphism of groups, the equality $g_{\sigma_{j}} \circ g_{\sigma_{i}}=g_{\sigma_{i} \sigma_{j}}$ holds; so we are done.
\end{proof}

\s

In our case, as ${\mathcal O}$ will be of genus zero, the automorphisms $g_{\sigma}$ will be suitable M\"obius transformations. Moreover, it will turn that $U<{\mathcal K}$, where
${\mathfrak S}_{3} \cong {\mathcal K}=\langle T(z)=1/z, L(z)=1/(z-1)\rangle<{\mathbb M}$, where ${\mathbb M}$ denotes the group of M\"obius transformations and ${\mathfrak S}_{3}$ denotes the symmetric group of order $6$. 

We will also need the following result in Huggins \cite{Huggins2}.

\s
\noindent
\begin{lemm}[Huggings \cite{Huggins2}]\label{lemahuggings}
Let $K$ be a subfield of $\mathbb C$.
Let $C$ be a genus zero non-singular projective algebraic curve and assume $C$ has a divisor $D$ of odd degree and rational over $K$. Then $C$ has $K$-rational points.
\end{lemm}

%%%%%%%%%%%%%%%%%%%%%%%
%%%%%%%%%%%%%%%%%%%%%%%
\section{Proof of part (1) of Theorem \ref{theo1}: Computation of ${\mathcal M}(C_{\lambda_{1},\lambda_{2}})$}

Set $\Omega=\{(z,w) \in ({\mathbb C}-\{0,1\})^{2}: z \neq w\}$. The group ${\mathbb G}$ acts a s a group of analytic automorphisms of $\Omega$. The space $\Omega$ is known as the Torelli space of the $5$-punctures Riemann sphere and $\Omega/{\mathbb G}$ turns out to be the moduli space of the $5$-punctured sphere, which is the same as the moduli space of classical Humbert curves as a consequence of the uniqueness of the classical Humbert group.  Next, we provide an explicit degree $5!$ map 
$$j:\Omega \to {\mathbb C}^{2}: (z,w) \mapsto (j_{1}(z,w), j_{2}(z,w))$$
so that $j(T(z,w))=j(z,w)$, for every $T \in {\mathbb G}$ and every $(z,w) \in \Omega$, and so that $j_{1}$ and $j_{2}$ are rational maps defined over ${\mathbb Z}$. That is, $j:\Omega \to j(\Omega)$ is a branched regular cover with ${\mathbb G}$ as its deck group; in particular, that $j(\Omega)$ is the moduli space of classical Humbert curves. For it, we consider the polynomial map $P(z,w)=(z^4,w^2)$ and the average rational map
$$j(z,w)=\sum_{T \in {\mathbb G}} P(T(z,w)) \in {\mathbb Q}(z,w).$$

By direct computation (done, for instance, with {\sc mathematica}) one may check that $j$ has degree $5!$. As $P$ and the transformations $T \in {\mathbb G}$ are defined over ${\mathbb Z}$, one can see that  $j$ is  as desired. The explicit forms of $j_{1}(z,w)$ and $j_{2}(z,w)$ are given as
$$j_{1}(z,w)=\frac{P_{1}(z,w)}{R_{1}(z,w)}$$
and
$$j_{2}(z,w)=\frac{P_{2}(z,w)}{R_{2}(z,w)},$$
where
$$R_{1}(z,w)=(z-1)^{4} z^{4}(z - w)^{4}(w-1)^{4} w^{4}$$
$$R_{2}(z,w)=(z-1)^{2} z^{2}(z - w)^{2}(w-1)^{2} w^{2}$$

\s
\noindent
$P_{1}(z,w)=
4 (2 z^8-8 z^9+12 z^{10}-8 z^{11}+6 z^{12}-8 z^{13}+12 z^{14}-8 z^{15}+
2 z^{16}-8 z^7 w+ 20 z^8 w-40 z^{10} w+16 z^{11}
w+16 z^{12} w-40 z^{13} w+
20 z^{15} w-8 z^{16} w+12 z^6 w^2-86 z^8 w^2+104 z^9 w^2+92 z^{10} w^2-112 z^{11} w^2+92 z^{12} w^2+104 z^{13} w^2-86 z^{14} w^2+12 z^{16} w^2-
8 z^5 w^3-40 z^6 w^3+104 z^7 w^3+84 z^8 w^3-416 z^9 w^3+56 z^{10} w^3+56 z^{11} w^3-416 z^{12} w^3+84 z^{13} w^3+104 z^{14} w^3-40 z^{15} w^3-
8 z^{16} w^3+6 z^4 w^4+16 z^5 w^4+92 z^6 w^4-416 z^7 w^4+439 z^8 w^4+448 z^9 w^4-176 z^{10} w^4+448 z^{11} w^4+439 z^{12} w^4-416 z^{13} w^4+
92 z^{14} w^4+16 z^{15} w^4+6 z^{16} w^4-8 z^3 w^5+16 z^4 w^5-112 z^5 w^5+56 z^6 w^5+448 z^7 w^5-1124 z^8 w^5-88 z^9 w^5-88 z^{10} w^5-
1124 z^{11} w^5+448 z^{12} w^5+56 z^{13} w^5-112 z^{14} w^5+16 z^{15} w^5-8 z^{16} w^5+12 z^2 w^6-40 z^3 w^6+92 z^4 w^6+56 z^5 w^6-176 z^6 w^6-
88 z^7 w^6+1522 z^8 w^6-720 z^9 w^6+1522 z^{10} w^6-88 z^{11} w^6-176 z^{12} w^6+56 z^{13} w^6+92 z^{14} w^6-40 z^{15} w^6+12 z^{16} w^6-
8 z w^7+104 z^3 w^7-416 z^4 w^7+448 z^5 w^7-88 z^6 w^7-720 z^7 w^7-388 z^8 w^7-388 z^9 w^7-720 z^{10} w^7-88 z^{11} w^7+448 z^{12} w^7-
416 z^{13} w^7+104 z^{14} w^7-8 z^{16} w^7+2 w^8+20 z w^8-86 z^2 w^8+84 z^3 w^8+439 z^4 w^8-1124 z^5 w^8+1522 z^6 w^8-388 z^7 w^8+
1182 z^8 w^8-388 z^9 w^8+1522 z^{10} w^8-1124 z^{11} w^8+439 z^{12} w^8+84 z^{13} w^8-86 z^{14} w^8+20 z^{15} w^8+2 z^{16} w^8-8 w^9+104 z^2 w^9-
416 z^3 w^9+448 z^4 w^9-88 z^5 w^9-720 z^6 w^9-388 z^7 w^9-388 z^8 w^9-720 z^9 w^9-88 z^{10} w^9+448 z^{11} w^9-416 z^{12} w^9+
104 z^{13} w^9-8 z^{15} w^9+12 w^{10}-40 z w^{10}+92 z^2 w^{10}+56 z^3 w^{10}-176 z^4 w^{10}-88 z^5 w^{10}+1522 z^6 w^{10}-720 z^7 w^{10}+1522 z^8
w^{10}-
88 z^9 w^{10}-176 z^{10} w^{10}+56 z^{11} w^{10}+92 z^{12} w^{10}-40 z^{13} w^{10}+12 z^{14} w^{10}-8 w^{11}+16 z w^{11}-112 z^2 w^{11}+56 z^3 w^{11}+448
z^4 w^{11}-
1124 z^5 w^{11}-88 z^6 w^{11}-88 z^7 w^{11}-1124 z^8 w^{11}+448 z^9 w^{11}+56 z^{10} w^{11}-112 z^{11} w^{11}+16 z^{12} w^{11}-8 z^{13} w^{11}+6
w^{12}+
16 z w^{12}+92 z^2 w^{12}-416 z^3 w^{12}+439 z^4 w^{12}+448 z^5 w^{12}-176 z^6 w^{12}+448 z^7 w^{12}+439 z^8 w^{12}-416 z^9 w^{12}+92 z^{10} w^{12}+
16 z^{11} w^{12}+6 z^{12} w^{12}-8 w^{13}-40 z w^{13}+104 z^2 w^{13}+84 z^3 w^{13}-416 z^4 w^{13}+56 z^5 w^{13}+56 z^6 w^{13}-416 z^7 w^{13}+84 z^8
w^{13}+
104 z^9 w^{13}-40 z^{10} w^{13}-8 z^{11} w^{13}+12 w^{14}-86 z^2 w^{14}+104 z^3 w^{14}+92 z^4 w^{14}-112 z^5 w^{14}+92 z^6 w^{14}+104 z^7 w^{14}-
86 z^8 w^{14}+12 z^{10} w^{14}-8 w^{15}+20 z w^{15}-40 z^3 w^{15}+16 z^4 w^{15}+16 z^5 w^{15}-40 z^6 w^{15}+20 z^8 w^{15}-8 z^9 w^{15}+2 w^{16}-
8 z w^{16}+12 z^2 w^{16}-8 z^3 w^{16}+6 z^4 w^{16}-8 z^5 w^{16}+12 z^6 w^{16}-8 z^7 w^{16}+2 z^8 w^{16})
$

\s
\noindent
$P_{2}(z,w)=4(2 z^4-4 z^5+6 z^6-4 z^7+2 z^8-4 z^3 w+2 z^4 w-4 z^5 w-4 z^6 w+2 z^7 w-4 z^8 w+6 z^2 w^2-4 z^3 w^2+23 z^4 w^2-16 z^5 w^2+
23 z^6 w^2-4 z^7 w^2+6 z^8 w^2-4 z w^3-4 z^2 w^3-16 z^3 w^3-6 z^4 w^3-6 z^5 w^3-16 z^6 w^3-4 z^7 w^3-4 z^8 w^3+2 w^4+2 z w^4+
23 z^2 w^4-6 z^3 w^4+30 z^4 w^4-6 z^5 w^4+23 z^6 w^4+2 z^7 w^4+2 z^8 w^4-4 w^5-4 z w^5-16 z^2 w^5-6 z^3 w^5-6 z^4 w^5-
16 z^5 w^5-4 z^6 w^5-4 z^7 w^5+6 w^6-4 z w^6+23 z^2 w^6-16 z^3 w^6+23 z^4 w^6-4 z^5 w^6+6 z^6 w^6-4 w^7+2 z w^7-4 z^2 w^7-
4 z^3 w^7+2 z^4 w^7-4 z^5 w^7+2 w^8-4 z w^8+6 z^2 w^8-4 z^3 w^8+2 z^4 w^8)$.

\s

Now, we proceed to compute ${\mathcal M}(C_{\lambda_{1},\lambda_{2}})$.
Let $\sigma \in K_{\lambda_{1},\lambda_{2}}$. As $C_{\lambda_{1},\lambda_{2}} \cong C_{\sigma(\lambda_{1}),\sigma(\lambda_{2})}$, there exists $T \in {\mathbb G}$ so that $T(\lambda_{1},\lambda_{2})=(\sigma(\lambda_{1}),\sigma(\lambda_{2}))$. Now, by the above explicit form of $j$, we have that $\sigma(j(\lambda_{1},\lambda_{2}))=j(\sigma(\lambda_{1}),\sigma(\lambda_{2}))=j \circ T (\lambda_{1},\lambda_{2})=j(\lambda_{1},\lambda_{2})$, from which we obtain that $ {\mathbb Q}(j_{i}(\lambda_{1},\lambda_{2}))<{\mathcal M}(C_{\lambda_{1},\lambda_{2}})$. In this way, 
${\mathbb Q}(j_{1}(\lambda_{1},\lambda_{2}),j_{2}(\lambda_{1},\lambda_{2}))<{\mathcal M}(C_{\lambda_{1},\lambda_{2}})$.

If $\sigma \in {\rm Aut}({\mathbb C}/{\mathbb Q})$ satisfies that $\sigma(j(\lambda_{1},\lambda_{2}))=j(\lambda_{1},\lambda_{2})$, then we have that 
$j(\lambda_{1},\lambda_{2})=\sigma(j(\lambda_{1},\lambda_{2}))=j(\sigma(\lambda_{1}),\sigma(\lambda_{2}))$. It follows that there exists some $T \in {\mathbb G}$ so that $T(\lambda_{1},\lambda_{2})=(\sigma(\lambda_{1}),\sigma(\lambda_{2}))$. In particular, $C_{\lambda_{1},\lambda_{2}}$ and $C_{\sigma(\lambda_{1}),\sigma(\lambda_{2})}$ are conformally equivalent Riemann surfaces and $\sigma \in K_{\lambda_{1},\lambda_{2}}$. This provides the inclusion
${\mathcal M}(C_{\lambda_{1},\lambda_{2}}) < {\mathbb Q}(j_{1}(\lambda_{1},\lambda_{2}),j_{2}(\lambda_{1},\lambda_{2}))$.

\s

%%%%%%%%%%%%%%%%%%%%%%%
%%%%%%%%%%%%%%%%%%%%%%%
\section{Proof of part (2) of Theorem \ref{theo1}: ${\mathcal M}(C_{\lambda_{1},\lambda_{2}})$ is a field of definition}
Let us denote by ${\mathcal O}_{\lambda_{1},\lambda_{2}}$ the orbifold whose underlying Riemann surface is the Riemann sphere $\widehat{\mathbb C}$ and its cone points are given by $\infty$, $0$, $1$, $\lambda_{1}$ and $\lambda_{2}$, each one of them with cone order equal to $2$. In this way, $C_{\lambda_{1},\lambda_{2}}/H={\mathcal O}_{\lambda_{1},\lambda_{2}}$ and
$\pi:C_{\lambda_{1},\lambda_{2}} \to {\mathcal O}_{\lambda_{1},\lambda_{2}}$ is defined by
$\pi([x_{1}:x_{2}:x_{3}:x_{4}:x_{5}])=-(x_{2}/x_{1})^{2}$.

The group ${\rm Aut}(C_{\lambda_{1},\lambda_{2}})/H={\rm Aut}_{orb}({\mathcal O}_{\lambda_{1},\lambda_{2}})$, the group of conformal orbifold automorphisms of ${\mathcal O}_{\lambda_{1},\lambda_{2}}$, is the finite subgroup of ${\mathbb M}$, the group of M\"obius transformations, that keeps invariant the set $\{\infty,0,1,\lambda_{1},\lambda_{2}\}$. As the finite  subgroups of ${\mathbb M}$, different from the trivial group, are either (i) cyclic groups, (ii) dihedral groups, (iii) the alternating group ${\mathcal A}_{4}$, (iv) the alternating group ${\mathcal A}_{5}$ and (v) the symmetric group ${\mathfrak S}_{4}$, it is not difficult to see that the only possibilities for ${\rm Aut}_{orb}({\mathcal O}_{\lambda_{1},\lambda_{2}})$ are given by 
either the trivial group, or the cyclic groups ${\mathbb Z}_{2}$, ${\mathbb Z}_{3}$, ${\mathbb Z}_{4}$, ${\mathbb Z}_{5}$ or the dihedral groups  ${\mathbb D}_{3}$, ${\mathbb D}_{5}$.

\subsection{}
Let us first assume that ${\rm Aut}(C_{\lambda_{1},\lambda_{2}})/H$ is not trivial nor isomorphic to ${\mathbb Z}_{2}$.
Under that assumption, $C_{\lambda_{1},\lambda_{2}}/Aut(C_{\lambda_{1},\lambda_{2}})={\mathcal O}_{\lambda_{1},\lambda_{2}}/{\rm Aut}_{orb}({\mathcal O}_{\lambda_{1},\lambda_{2}})$ is an orbifold with triangular signature, that is, the underlying Riemann surface structure is $\widehat{\mathbb C}$ and it has exactly $3$ cone points. It follows that $C_{\lambda_{1},\lambda_{2}}$ is a quasi-platonic surface. As quasi-platonic surfaces can be defined over their field of moduli \cite{Wolfart}, we are done in this case.

\subsection{}
Let us now assume that ${\rm Aut}(C_{\lambda_{1},\lambda_{2}})/H \cong {\mathbb Z}_{2}$; equivalent that ${\rm Aut}_{orb}(S/H) \cong {\mathbb Z}_{2}$. In this way, the subgroup of ${\mathbb M}$ keeping invariant the set of cone points $\infty$, $0$, $1$, $\lambda_{1}$ and $\lambda_{2}$ is generated by an elliptic transformation $A \in {\mathbb M}$ of order two. We must have that $A(p_{1})=p_{1}$, 
$A(p_{2})=p_{3}$ and $A(p_{4})=p_{5}$, where
$\{p_{1},p_{2},p_{3},p_{4},p_{5}\}=\{\infty,0,1,\lambda_{1},\lambda_{2}\}$. Let $T \in {\mathbb M}$ so that $T(p_{1})=1$, $T(p_{2})=0$ and $T(p_{3})=\infty$. Then $T \circ A \circ T^{-1}(z)=1/z$. In particular, $T(p_{5})=1/T(p_{4})$.
So, up to the action of ${\mathbb G}$ at the set of parameters $(\lambda_{1},\lambda_{2})$, we may assume that $\lambda_{2}=\lambda_{1}^{-1}$ and that ${\rm Aut}_{orb}({\mathcal O}_{\lambda_{1},\lambda_{2}})=\langle A(z)=1/z \rangle$. In this case, a lifting of $A$ is the automorphism of $C_{\lambda_{1},\lambda_{2}}$ given as
$$\widehat{A}([x_{1}:x_{2}:x_{3}:x_{4}:x_{5}])=[x_{2}:x_{1}:x_{3}:\sqrt{\lambda_{1}}x_{5}:\sqrt{\lambda_{2}}x_{4}]$$
and ${\rm Aut}(C_{\lambda_{1},\lambda_{2}})=\langle H, \widehat{A}\rangle$.

In the rest, we set $\lambda=\lambda_{1}$, $K_{\lambda}:=K_{\lambda_{1},\lambda_{2}}$, $C_{\lambda}:=C_{\lambda_{1},\lambda_{2}}$ and ${\mathcal O}_{\lambda}:={\mathcal O}_{\lambda_{1},\lambda_{2}}$. Let us consider the map $Q:\widehat{\mathbb C} \to \widehat{\mathbb C}$ defined by 
$$
Q(z)=\frac{-\lambda}{(1+\lambda)^{2}} (z+1/z+2)+1.
$$

Then, $Q$ defines a degree two branched cover with $\langle A\rangle$ as deck group. In this way, it can be checked that 
$P=Q \circ \pi:C_{\lambda} \to \widehat{\mathbb C}$ defines a regular branched cover with ${\rm Aut}(C_{\lambda})$ as its deck group. Note that 
$$
Q(\infty)=Q(0)=\infty,\; Q(\lambda)=Q(\lambda^{-1})=0,\; Q(1)=\rho_{\lambda}=\left( \frac{\lambda-1}{\lambda+1} \right)^{2}, \; Q(-1)=1.
$$

In this way, the cone points of orbifold ${\mathcal O}_{\lambda}=C_{\lambda}/{\rm Aut}(C_{\lambda})$ are given by $\infty$, $0$, $1$ (the three of them with cone order equal to $2$) and $\rho_{\lambda}$ with cone order equal $4$.

As we are assuming that ${\rm Aut}_{orb}({\mathcal O}_{\lambda})$ is isomorphic to ${\mathbb Z}_{2}$, we must have that
$\rho_{\lambda} \notin \{-1,0,1/2,2,(1 \pm i \sqrt{3})/2\}$.

For each $\sigma \in K_{\lambda}$ there exists a conformal isomorphism
$f_{\sigma}:C_{\lambda} \to C_{\sigma(\lambda)}$. By the uniqueness of $H$, in both ${\rm Aut}(C_{\lambda})$ and ${\rm Aut}(C_{\sigma(\lambda)})$, it follows that there exists a (uniquely determined by $\sigma$) M\"obius transformation $g_{\sigma} \in {\mathbb M}$ so that $P^{\sigma} \circ f_{\sigma}= g_{\sigma} \circ P$. Note that $P^{\sigma}=Q^{\sigma} \circ \pi^{\sigma}=Q^{\sigma} \circ \pi$, where
$$
Q^{\sigma}(z)=\frac{-\sigma(\lambda)}{(1+\sigma(\lambda))^{2}} (z+1/z+2)+1.
$$

Also, as $g_{\sigma}$ should preserve the cone points and the cone orders, we have that
$$
g_{\sigma}\left(\{\infty,0,1\}\right)=\{\infty,0,1\}
$$
$$
g_{\sigma}(\rho_{\lambda})= \rho_{\sigma(\lambda)}.
$$

In particular,  $g_{\sigma} \in \langle T(z)=1/z, L(z)=1/(1-z)\rangle \cong {\mathfrak S}_{3}$. It follows from Lemma \ref{homo} that 
$\Phi:K_{\lambda} \to {\mathbb M}$, defined as $\Phi(\sigma)=g_{\sigma}^{-1}$, is a homomorphism of groups.

The uniqueness of $g_{\sigma}$ ensures that the collection $\{g_{\sigma}: \sigma \in K_{\lambda}\}$ satisfies the conditions of Weil's theorem \cite{Weil}. It follows the existence of a conformal isomorphism $R:\widehat{\mathbb C} \to Z$, where $Z$ is defined over ${\mathcal M}(C_{\lambda})$, and so that $R=R^{\sigma} \circ g_{\sigma}$, for every $\sigma \in K_{\lambda}$. 

In this way, $\widehat{P}=R \circ P:C_{\lambda} \to Z$ is a regular branched cover with ${\rm Aut}(C_{\lambda})$ as its deck group so that, for every $\sigma \in K_{\lambda}$ it holds that $\widehat{P}=\widehat{P}^{\sigma} \circ f_{\sigma}$.

By Lemma \ref{calculo}, in order to prove that ${\mathcal M}(C_{\lambda})$ is a field of definition for $C_{\lambda}$, we only need to prove the existence of a point $r \in \widehat{\mathbb C}-\{\infty,0,1,\rho_{\lambda}\}$ such that, for every $\sigma \in K_{\lambda}$ it holds the equality  $g_{\sigma}(r)=\sigma(r)$. Next, using Lemma \ref{genera}, we proceed to compute such a value of $r \in {\mathbb C}-\{0,1,\rho_{\lambda}\}$ for each of the different possibilities for $U$.
\begin{enumerate}
\item $U=\{I\}$. In this case, we may choose for $r$ any rational number different from $0$, $1$ and $\rho_{\lambda}$.

\item $U \cong {\mathbb Z}_{2}$. In this case, $U$ is generated by one of the following transformations: $T_{1}(z)=1/z$ or $T_{2}(z)=1-z$ or $T_{3}(z)=z/(z-1)$. If $U$ is generated by $T_{1}$, then we may take $r=-1$. If $U$ is generated by $T_{2}$, then we may take $r=1/2$. If the generator of $U$ is $T_{3}$, then we may take $r=2$.

\item $U \cong {\mathbb Z}_{3}$. In this case, $U$ is generated by the transformation $L_{1}(z)=1/(1-z)$. By Lemma \ref{genera}, we only need to find $r$ so that $\sigma(r)=1/(1-r)$ and $\sigma(\rho_{\lambda})=1/(1-\rho_{\lambda})$. If we set 
$$
r=\frac{a_{0}+a_{1}\rho_{\lambda}}{(a_{0}+a_{1})-a_{0}\rho_{\lambda}},
$$
where $a_{0},a_{1} \in {\mathbb Q}$, then $\sigma(r)=g_{\sigma}(r)$.

\item $U={\mathfrak S}_{3}$. In this case, by Lemma \ref{genera}, we only need to find $r$ so that $\sigma(r)=1/(1-r)$, $\tau(r)=1/r$, $\sigma(\rho_{\lambda})=1/(1-\rho_{\lambda})$ and $\tau(\rho_{\lambda})=1/\rho_{\lambda}$. In this case we may take 
$$
r=\frac{\rho_{\lambda}(\rho_{\lambda}-2)}{1-2\rho_{\lambda}}.
$$
\end{enumerate}

\s

\subsection{}
Let us now assume that ${\rm Aut}(C_{\lambda_{1},\lambda_{2}})=H$. We have that $C_{\lambda_{1},\lambda_{2}}/{\rm Aut}(C_{\lambda_{1},\lambda_{2}})$ is the Riemann sphere with the cone points given by $p_{1}=\infty$, $p_{2}=0$, $p_{3}=1$, $p_{4}=\lambda_{1}$ and $p_{5}=\lambda_{2}$, each of them with cone order equal to $2$. Let us consider the positive divisor of degree five $D=p_{1}+p_{2}+p_{3}+p_{4}+p_{5}$. For each $\sigma \in K_{\lambda_{1},\lambda_{2}}$ there is a conformal isomorphism $f_{\sigma}:C_{\lambda_{1},\lambda_{2}} \to C_{\sigma(\lambda_{1}),\sigma(\lambda_{2})}$. Now, as in the previous case, there is a M\"obius transformation $g_{\sigma}$ (uniquely determined by $\sigma$) so that $\pi \circ f_{\sigma}=g_{\sigma} \circ \pi$. The set of cone points of $C_{\sigma(\lambda_{1}),\sigma(\lambda_{2})}/{\rm Aut}(C_{\sigma(\lambda_{1}),\sigma(\lambda_{2})})$ is given by 
$\{\infty,0,1,\sigma(\lambda_{1}),\sigma(\lambda_{2})\}$. On the other hand, as $g_{\sigma}$ should send the cone points to cone points, it follows that $\{g_{\sigma}(p_{1}),\ldots,g_{\sigma}(p_{5})\}=\{\infty,0,1,\sigma(\lambda_{1}),\sigma(\lambda_{2})\}$. As the collection $\{g_{\sigma}\}_{\sigma \in K_{\lambda_{1},\lambda_{2}}}$ satisfies Weil's theorem \cite{Weil}, it follows the existence of a conformal isomorphism $R:\widehat{\mathbb C} \to Z$, where $Z$ is defined over ${\mathcal M}(C_{\lambda_{1},\lambda_{2}})$, and so that $R=R^{\sigma} \circ g_{\sigma}$, for every $\sigma \in K_{\lambda_{1},\lambda_{2}}$. The divisor $R(D)=r(p_{1})+r(p_{2})+r(p_{3})+r(p_{4})+r(p_{5})$ now satisfies that $R(D)^{\sigma}=R(D)$, for every $\sigma \in K_{\lambda_{1},\lambda_{2}}$. In this way, the genus zero curve $Z$ has a divisor of odd degree defined over ${\mathcal M}(C_{\lambda_{1},\lambda_{2}})$. It now follows from Lemma \ref{lemahuggings} that $Z$ has ${\mathcal M}(C_{\lambda_{1},\lambda_{2}})$-rational points. At this point, these points may be branched values for the branched cover $R \circ \pi:C_{\lambda_{1},\lambda_{2}} \to Z$. Now, as $Z \cong _{{\mathcal M}(C_{\lambda_{1},\lambda_{2}})}{\mathbb P}^{1}_{{\mathcal M}(C_{\lambda_{1},\lambda_{2}})}$, it follows that $Z$ has ${\mathcal M}(C_{\lambda_{1},\lambda_{2}})$-rational points off the branch locus. Now the result follows again from Theorem \ref{canonico}.

%%%%%%%%%%%%%%%%%%%%
%%%%%%%%%%%%%%%%%%%%

\end{document}